\documentclass[12pt,a4paper,reqno]{amsart}
\usepackage{amsmath}
\usepackage{amsfonts}
\usepackage{amssymb}
\usepackage{amsthm}
\usepackage{enumerate}
\usepackage{centernot}
\usepackage{fullpage}
\usepackage{mathrsfs}
\usepackage{graphicx}

\newcommand{\upperRomannumeral}[1]{\uppercase\expandafter{\romannumeral#1}}
\newcommand{\lowerRomannumeral}[1]{\lowercase\expandafter{\romannumeral#1}}

\theoremstyle{plain}

  \newtheorem{proposition}[]{Proposition}
  \newtheorem{lemma}[]{Lemma}
  \newtheorem{theorem}[]{Theorem}
  \newtheorem{corollary}[]{Corollary}
  
  \newtheorem{remark}[]{Remark}
  \newtheorem*{remark*}{Remark}

\title{A note on exponential decay in the\\ random field Ising model}
\author{Federico Camia}
\address{Division of Science, NYU Abu Dhabi, Saadiyat Campus, Abu Dhabi, UAE \& Department of Mathematics, VU Amsterdam, De Boelelaan 1081a, 1081 HV Amsterdam, the Netherlands}
\email{federico.camia@nyu.edu}
\author{Jianping Jiang}
\address{NYU-ECNU Institute of Mathematical Sciences at NYU Shanghai, 3663 Zhongshan
Road North, Shanghai 200062, China.}
\email{jjiang@nyu.edu}
\author{Charles M. Newman}
\address{Courant Institute of Mathematical Sciences, New York University,
251 Mercer st, New York, NY 10012, USA, \& NYU-ECNU Institute of Mathematical
Sciences at NYU Shanghai, 3663 Zhongshan Road North, Shanghai 200062, China.}
\email{newman@cims.nyu.edu}
\begin{document}
\begin{abstract}
For the two-dimensional random field Ising model (RFIM) with bimodal (i.e., two-valued) external field, we prove exponential decay of correlations either (\lowerRomannumeral{1}) when the temperature is larger than the critical temperature of the Ising model without external field and the magnetic field strength is small or (\lowerRomannumeral{2}) at any temperature when the magnetic field strength is sufficiently large. Unlike previous work on exponential decay, our approach is not based on cluster expansions but rather on arguably simpler methods; these combine an analysis of the Kert\'{e}sz line and a coupling of Ising measures (and also their random cluster representations) with different boundary conditions. We also show similar but weaker results for the RFIM with a general field distribution and in any dimension.
\end{abstract}
\maketitle
\section{Introduction and two-dimensional results}
\subsection{Overview}
The random field Ising model was introduced by Imry and Ma \cite{IM75} as a model of a disordered system. They predicted that the RFIM has a phase transition if and only if the dimension $d\geq3$. Bricmont and Kupiainen \cite{BK88} proved the existence of a phase transition for $d\geq3$, and Aizenman and Wehr \cite{AW90} proved for any temperature uniqueness of the Gibbs state when $d=2$.

Exponential decay of correlations for the RFIM when $d\geq2$ for (\lowerRomannumeral{1}) high temperature and for (\lowerRomannumeral{2}) any temperature with sufficiently large magnetic field  strength was proved in \cite{FI84, Ber85, DKP95} using cluster expansion methods. A rate of decay (an iterated logarithm in distance) for the RFIM when $d=2$ for all temperatures was proved in \cite{Cha18}. The results in \cite{FI84, Ber85, DKP95} do not provide any quantitative information about the regions of exponential decay (in terms of temperature and magnetic field strength). In this paper (see Theorems \ref{thmcor1}-\ref{thmABG} below), we prove exponential decay in some specified regions based on an analysis of the Kert\'{e}sz line \cite{Ker89,BGLRS08,RW08} and by using couplings of Ising random cluster measures. In addition to giving more detailed information about the location of exponential decay regions, the other main contribution of this paper is that it provides a different (and arguably simpler) proof of exponential decay than in previous work.

The Kert\'{e}sz line (or curve)  is defined according to the existence or not of an infinite cluster on $\mathbb{Z}^d$ in the random cluster representation of the Ising or $q$-state Potts model. For $d=2$ and $q=2$, we show (see part (\lowerRomannumeral{1}) of Theorem \ref{thmKerIsing} below) that this line is located at a magnetic field strength that is strictly positive when the temperature is strictly larger than the critical temperature. For $d\geq 1$ and $q\geq 2$, we show (see Theorem \ref{thmKer} or part (\lowerRomannumeral{2}) of Theorem \ref{thmKerIsing} below) strict positivity when the temperature is large. We refer to \cite{BGLRS08,RW08} and references therein for more information about the Kert\'{e}sz line.

Let us mention that a major open problem in the RFIM is to determine the true decay rate when both the temperature and magnetic field strength are low. In particular for $d=2$, there seem to be competing predictions about polynomial versus exponential decay --- see the discussion near the end of Section 1 of \cite{Cha18}.

The organization of the paper is as follows. In the rest of this section we provide some definitions and present our main two-dimensional results. In Section \ref{gendim}, we state our results for general dimension and discuss the main ideas behind the proofs. In Section \ref{secKer}, we give some results about the location of the Kert\'{e}sz line (see Theorems \ref{thmKer} and \ref{thmKerIsing} there). Sections \ref{secMixing} and \ref{secCor} are devoted to the proofs of our main results, namely Theorems \ref{thmSBSDB}-\ref{thmABG} and Corollary~\ref{cor2} below.

\subsection{Definitions}
Let $\Lambda\subseteq \mathbb{Z}^d$ be a finite subset, and denote by $\partial_{ex}\Lambda$ the set of vertices in $\mathbb{Z}^d\setminus\Lambda$ which have a nearest neighbor in $\Lambda$. The classical Ising model on $\Lambda$ at inverse temperature $\beta$ with boundary condition $\eta \in \{-1,+1\}^{\partial_{ex} \Lambda}$ and external field $\vec{\mathcal{H}}\in \mathbb{R}^{\Lambda}$ is defined by the probability measure $P^{\vec{\mathcal{H}}}_{\Lambda,\eta,\beta}$ on $\{-1,+1\}^{\Lambda}$ such that for each $\sigma\in \{-1,+1\}^{\Lambda}$,
\begin{equation}\label{eqIsingH}
P^{\vec{\mathcal{H}}}_{\Lambda,\eta,\beta}(\sigma)=\frac{1}{Z^{\vec{\mathcal{H}}}_{\Lambda,\eta,\beta}}\exp{\left[\beta\sum_{\{u,v\}}\sigma_u\sigma_v+\beta\sum_{\{u,v\}: u\in\Lambda, v\in\partial_{ex} \Lambda}\sigma_u\eta_v+\sum_{u\in\Lambda}\mathcal{H}_u\sigma_u\right]},
\end{equation}
where the first sum is over all nearest neighbor pairs in $\Lambda$, and $Z^{\vec{\mathcal{H}}}_{\Lambda,\eta,\beta}$ is the partition function that makes \eqref{eqIsingH} a probability measure.

Suppose that $\vec{\mathcal{H}}:=\{HH_u,u\in\Lambda\}$ where the $H_u$ are i.i.d. random variables with a common distribution $\nu$ of mean~$0$ and variance $1$, and $H\geq 0$. The resulting $P^{\vec{\mathcal{H}}}_{\Lambda,\eta,\beta}$, which we now denote by $P^{\vec{H}}_{\Lambda,\eta,\beta,H}$, is a random probability measure. This is known as a \textbf{random field Ising model}. We mainly consider in this paper the special cases of the \textbf{bimodal field} (i.e., $Prob(H_u=+1)=Prob(H_u=-1)=1/2$) and the \textbf{Gaussian field} (i.e., $H_u\overset{d}{=}N(0,1)$). We also occasionally consider the RFIM with a more general common distribution $\nu$ for the $H_u$'s (see Remark \ref{remgen} following Theorem \ref{thmABG} below). In the rest of the paper, $\vec{H}$ denotes the random field while $\vec{h}\in \mathbb{R}^{\Lambda}$ denotes a fixed field configuration;  $P^{\vec{h}}_{\Lambda,\eta,\beta,H}$ is often called the quenched distribution. We denote by $\langle \cdot \rangle^{\vec{h}}_{\Lambda,\eta,\beta,H}$  the expectation with respect to $P^{\vec{h}}_{\Lambda,\eta,\beta,H}$. Let $\beta_c(d)$ be the critical inverse temperature of the Ising model on $\mathbb{Z}^d$ without external field (i.e., with $H=0$). Define $\beta_P(d)$ by
\begin{equation}
1-e^{-2\beta_P(d)}=p_c^b(d),
\end{equation}
where $p_c^b(d)$ is the critical probability for independent Bernoulli bond percolation on $\mathbb{Z}^d$.
Let $TV(\cdot,\cdot)$ denote the total variation distance between probability measures. Denote by $|\cdot|$ the Euclidean distance and let $d(U,V):=\inf_{x\in U, y\in V}|x-y|$ denote the distance between two sets $U,V\subseteq\mathbb{R}^d$. Let $\Lambda_L:=[-L,L]^d$ be the box of side length $2L$ centered at the origin.

\subsection{Two-dimensional results}
One of the main results in two dimensions is:
\begin{theorem}\label{thmcor1}
Consider the RFIM with bimodal field and with $d=2$. For all $0\leq \beta<\beta_c(2)$, there exists $H_1(\beta)> 0$ such that for each $H\in[0,H_1(\beta))$,
\begin{equation}
\sup_{\eta,\eta^{\prime}\in\{-1,+1\}^{\partial_{ex} \Lambda_L}}\left[\langle \sigma_0 \rangle^{\vec{h}}_{\Lambda_L,\eta,\beta,H}-\langle \sigma_0 \rangle^{\vec{h}}_{\Lambda_L,\eta^{\prime},\beta,H}\right] \leq C_2(\beta,H) e^{-C_1(\beta, H)L}
\end{equation}
for each realization $\vec{h}\in\{-1,+1\}^{\Lambda}$, where $C_1(\beta,H),C_2(\beta,H)\in (0,\infty)$ depend only on $\beta, H$. Here
\begin{equation}
 H_1(\beta) \text{ is }\begin{cases}
      \infty, & \beta\in[0,\beta_P(2)], \\
      \in(0,\infty), &\beta\in(\beta_P(2),\beta_c(2)).
   \end{cases}
\end{equation}
\end{theorem}

For $\Delta \subseteq \Lambda \subseteq \mathbb{Z}^d$ and $\sigma \in \{-1,+1\}^{\Lambda}$, we denote by $\sigma_{\Delta}$ the restriction of $\sigma$ to $\Delta$. We will actually prove the following stronger result (of which Theorem \ref{thmcor1} is a special case):

\begin{theorem}\label{thmSBSDB}
Consider the RFIM with bimodal field and with $d=2$. Let $\Delta\subseteq\Lambda$ be finite subsets of $\mathbb{Z}^2$. For all $0\leq \beta<\beta_c(2)$, for each $H\in[0,H_1(\beta))$,
\begin{equation}\label{eqmixingB21}
\sup_{\eta,\eta^{\prime}\in\{-1,+1\}^{\partial_{ex} \Lambda}}TV\left(P^{\vec{h}}_{\Lambda,\eta,\beta,H}(\sigma_{\Delta}\in \cdot), P^{\vec{h}}_{\Lambda,\eta^{\prime},\beta,H}(\sigma_{\Delta}\in \cdot)\right) \leq \sum_{x\in\partial_{ex}\Delta, y\in\partial_{ex}\Lambda}e^{-C_1(\beta,H)|x-y|}
\end{equation}
for each realization $\vec{h}\in\{-1,+1\}^{\Lambda}$, where $H_1(\beta)$ is as in Theorem \ref{thmcor1}.
\end{theorem}

Equations \eqref{eqmixingB21} is usually called the \textbf{weak mixing} property (see \cite{Ale98}). It is easy to see that the FKG lattice property and the weak mixing property together imply that there exists a unique infinite volume limit of $P^{\vec{h}}_{\Lambda,\eta,\beta,H}$ as $\Lambda\rightarrow \mathbb{Z}^d$ which does not depend on the choice of $\eta$. We denote this infinite volume limit by $P^{\vec{h}}_{\mathbb{Z}^d,\beta,H}$ and let $\langle \cdot \rangle^{\vec{h}}_{\mathbb{Z}^d,\beta,H}$ be its expectation. In \cite{Ale98}, a weak mixing property for such infinite volume measures is also defined; Theorem \ref{thmSBSDB} says that $P^{\vec{h}}_{\mathbb{Z}^d,\beta,H}$ has this weak mixing property for $(\beta,H)$ in the region described in the theorem. As another consequence of Theorem \ref{thmSBSDB}, we have the following exponential decay of the truncated two-point function. We remark that similar exponential decay occurs for the RFIMs and $(\beta,H)$ regions described in Theorems \ref{thmSBLDB}-\ref{thmABG} (and Remark \ref{remgen}) below.

\begin{corollary}\label{cor2}
Consider the RFIM with bimodal field and with $d=2$. For all $0\leq \beta<\beta_c(2)$, there exists $H_1(\beta)>0 $ such that for each $H\in[0,H_1(\beta))$,
\begin{equation}\label{eqttdecayG2}
0\leq \langle \sigma_x\sigma_y\rangle^{\vec{h}}_{\mathbb{Z}^d,\beta,H}-\langle \sigma_x\rangle^{\vec{h}}_{\mathbb{Z}^d,\beta,H}\langle \sigma_y\rangle^{\vec{h}}_{\mathbb{Z}^d,\beta,H} \leq C_{7}(\beta,H) e^{-C_{1}(\beta,H)|x-y|/2},~\forall x,y\in\mathbb{Z}^d
\end{equation}
for each realization $\vec{h}\in\{-1,+1\}^{\mathbb{Z}^d}$, where $C_{7}(\beta,H)\in (0,\infty)$ depends only on $\beta,H$.
\end{corollary}

\section{Results for general dimension}\label{gendim}
\subsection{Results for general dimension}
For the RFIM with bimodal field and general $d$, we have:
\begin{theorem}\label{thmSBLDB}
Consider the RFIM with bimodal field and with $d\geq 1$. Let $\Delta\subseteq\Lambda$ be finite subsets of $\mathbb{Z}^d$.
 For certain $\beta\in[0, \beta_c(d))$, as indicated below, there exists $\tilde{H}_1(\beta)> 0$ such that for each $H\in[0,\tilde{H}_1(\beta))$,
\begin{equation}\label{eqmixingB22}
\sup_{\eta,\eta^{\prime}\in\{-1,+1\}^{\partial_{ex} \Lambda}}TV\left(P^{\vec{h}}_{\Lambda,\eta,\beta,H}(\sigma_{\Delta}\in \cdot), P^{\vec{h}}_{\Lambda,\eta^{\prime},\beta,H}(\sigma_{\Delta}\in \cdot)\right) \leq \sum_{x\in\partial_{ex}\Delta, y\in\partial_{ex}\Lambda}e^{-C_1(\beta,H)|x-y|}
\end{equation}
for each realization $\vec{h}\in\{-1,+1\}^{\Lambda}$, where $C_1(\beta,H)\in (0,\infty)$ depends only on $\beta, H$ (and $d$). Here
\begin{equation}
 \tilde{H}_1(\beta)\text{ is }\begin{cases}
      \infty, & \beta\in[0,\beta_P(d)],\\
      \in(0,\infty), & \beta\in(\beta_P(d),\beta_P(d)+\epsilon_d],\\
      \in[0,\infty), & \beta \in (\beta_P(d)+\epsilon_d, \beta_c(d)),\\
   \end{cases}
\end{equation}
where $\epsilon_d>0$.
\end{theorem}

For the RFIM with Gaussian (or other) field and $\beta\in[0,\beta_P(d))$, we have:
\begin{theorem}\label{thmSBG}
Consider the RFIM with Gaussian field (or with any common distribution $\nu$) and with $d\geq1$. Let $\Delta\subseteq\Lambda$ be finite subsets of $\mathbb{Z}^d$. For all $0\leq \beta<\beta_P(d)$ and $H>0$,
\begin{equation}\label{eqmixingG2}
\sup_{\eta,\eta^{\prime}\in\{-1,+1\}^{\partial_{ex} \Lambda}}TV\left(P^{\vec{h}}_{\Lambda,\eta,\beta,H}(\sigma_{\Delta}\in \cdot), P^{\vec{h}}_{\Lambda,\eta^{\prime},\beta,H}(\sigma_{\Delta}\in \cdot)\right) \leq \sum_{x\in\partial_{ex}\Delta, y\in\partial_{ex}\Lambda}e^{-C_3(\beta)|x-y|}
\end{equation}
for each realization $\vec{h}\in\mathbb{R}^{\Lambda}$, where $C_3(\beta)\in (0,\infty)$ depends only on $\beta$ (and $d$).
\end{theorem}

\begin{remark}
Theorem \ref{thmSBG} extends (by the same proof) to \textbf{any} deterministic $\vec{h}\in\mathbb{R}^{\Lambda}$.
\end{remark}

For the RFIM with $\beta\geq0$ and $H$ large, we have the following two theorems.
\begin{theorem}\label{thmABB}
Consider the RFIM with bimodal field and with $d\geq 1$. Let $\Delta\subseteq\Lambda$ be finite subsets of $\mathbb{Z}^d$.
For all $0\leq \beta<\infty$, there exists $H_2(\beta)>0$ such that for each $H>H_2(\beta)$,
\begin{equation}\label{eqmixingB1}
\sup_{\eta,\eta^{\prime}\in\{-1,+1\}^{\partial_{ex} \Lambda}}TV\left(P^{\vec{h}}_{\Lambda,\eta,\beta,H}(\sigma_{\Delta}\in \cdot), P^{\vec{h}}_{\Lambda,\eta^{\prime},\beta,H}(\sigma_{\Delta}\in \cdot)\right)\leq \sum_{x\in\partial_{ex}\Delta, y\in\partial_{ex}\Lambda}e^{-C_4|x-y|}
\end{equation}
for each realization $\vec{h}\in\{-1,+1\}^{\Lambda}$, where $C_4\in (0,\infty)$ is a constant (depending only on~$d$).
\end{theorem}

\begin{theorem}\label{thmABG}
Consider the RFIM with Gaussian field and with $d\geq 1$.  Let $\Delta\subseteq\Lambda$ be finite subsets of $\mathbb{Z}^d$. For all $0\leq \beta<\infty$, there exists $H_3(\beta)>0$ such that for each $H>H_3(\beta)$,
\begin{equation}\label{eqmixingG1}
\sup_{\eta,\eta^{\prime}\in\{-1,+1\}^{\partial_{ex} \Lambda}}TV\left(P^{\vec{h}}_{\Lambda,\eta,\beta,H}(\sigma_{\Delta}\in \cdot), P^{\vec{h}}_{\Lambda,\eta^{\prime},\beta,H}(\sigma_{\Delta}\in \cdot)\right)\leq \sum_{x\in\partial_{ex}\Delta, y\in\partial_{ex}\Lambda}C_6(x,\vec{h})e^{-C_5|x-y|}
\end{equation}
for almost all realizations $\vec{h}\in\mathbb{R}^{\mathbb{Z}^d}$ of $\vec{H}$, where $C_6(x,\vec{h})\in(0,\infty)$ depends only on $x$ and $\vec{h}$ (and $d$) and $C_5\in (0,\infty)$ is a constant (depending only on $d$).
\end{theorem}

\begin{remark}\label{remgen}
Theorem \ref{thmABG} extends (by essentially the same proof) to the RFIM with a general common distribution $\nu$ satisfying
\begin{equation}
\nu(\{0\})=Prob(H_x=0)<p_c^s(d),
\end{equation}
where  $p_c^s(d)$ is the critical probability for independent Bernoulli site percolation on $\mathbb{Z}^d$.
\end{remark}

\begin{remark}\label{rembound}
The proofs of Theorem \ref{thmABB} (see Lemma \ref{lemCoupleLF}) and Theorem \ref{thmABG} (see \eqref{eqTopen} and \eqref{eqsignSG}) show the following: $H_2(\beta)$ can be any number satisfying
\begin{equation}
1-\Big(\frac{e^{-2d\beta+H_2(\beta)}}{e^{-2d\beta+H_2(\beta)}+e^{-2d\beta-H_2(\beta)}}\Big)^2<p_c^s(d)
\end{equation}
and $H_3(\beta)$ can be any number such that for some $\delta>0$
\begin{equation}
Prob(|H_x|<\delta)+1-\Big(\frac{e^{-2d\beta+H_3(\beta)\delta}}{e^{-2d\beta+H_3(\beta)\delta}+e^{-2d\beta-H_3(\beta)\delta}}\Big)^2<p_c^s(d).
\end{equation}
\end{remark}

\subsection{Ideas of the proofs}
The proofs of Theorems \ref{thmcor1}-\ref{thmSBG} are based first on nontriviality of the Kert\'{e}sz line for FK percolation when $\beta<\beta_c$ (see Theorem \ref{thmKerIsing}) and second on exponential decay in the homogeneous Ising model with constant magnetic field under the Kert\'{e}sz line (see Theorem \ref{thmPCdecay}). The main idea for the proofs of Theorems \ref{thmABB}-\ref{thmABG} is that each Ising spin will follow with high probability the sign of its external field when the field strength is large. The coupling of Ising measures (and their random cluster representations) with different boundary conditions also plays an important role --- see Lemmas \ref{lemcoupleRC}, \ref{lemcoupleIsing} and \ref{lemCoupleLF} .

\section{Location of the Kert\'{e}sz line}\label{secKer}
In this section, we consider the $q$-state Potts model and random cluster models. Note that Potts models are generalizations of the Ising model. We show that the Kert\'{e}sz line is located at a strictly positive magnetic field strength when (\lowerRomannumeral{1}) temperature is larger than the critical temperature for $d=q=2$ and when (\lowerRomannumeral{2}) temperature is sufficiently large for general $d\geq 1$ and $q>1$.

For any $q\in \{2,3,\dots\}$ and finite $\Lambda\subseteq \mathbb{Z}^d$,  the Potts model at inverse temperature $\beta$ on $\Lambda$ with boundary condition $\eta\in\{1,2,\dots,q\}^{\partial_{ex} \Lambda}$ and external field $H\geq0$ is defined by the probability measure $P_{\Lambda,\eta,\beta,H}$ on $\{1,2,\dots,q\}^{\Lambda}$ such that for each $\sigma\in\{1,2,\dots,q\}^{\Lambda}$,
\begin{equation}\label{eqPottsdef}
P_{\Lambda,\eta,\beta,H}(\sigma)=\frac{1}{Z_{\Lambda,\eta,\beta,H}}\exp{\left[2\beta\sum_{\{u,v\}}\delta_{\sigma_u,\sigma_v}+2\beta\sum_{e=\{u,v\}:u\in\Lambda, v\in\partial_{ex}\Lambda}\delta_{\sigma_u,\eta_v}+2H\sum_{u\in \Lambda}\delta_{\sigma_u,1}\right]},
\end{equation}
where the first sum is over all nearest neighbor pairs in $\Lambda$, $Z_{\Lambda,\eta,\beta,H}$ is the partition function and $\delta_{i,j}$ is the Kronecker delta. Note that our parametrization differs from the usual one (see, e.g., Section 1.3 of \cite{Gri06}) by a factor of $2$ for both $\beta$ and $H$; with this choice, the Potts model with $q=2$ corresponds to the Ising model with inverse temperature $\beta$ and constant external field $H$. In this section we only consider constant external field, i.e., each vertex has the same external field of strength $H$. 

Before we define the random cluster model, we need some notation. With vertex set $\mathbb{Z}^d$, we write $\mathbb{E}^d$ for the set of nearest neighbor edges of $\mathbb{Z}^d$. For $\Lambda\subseteq \mathbb{Z}^d$, define $\Lambda^C:=\mathbb{Z}^d\setminus\Lambda$,
\begin{equation}
\partial_{in}\Lambda:=\{z\in \mathbb{Z}^d: z\in \Lambda, z \text{ has a nearest neighbor in }\Lambda^C \},
\end{equation}

\begin{equation}
\partial_{ex}\Lambda:=\{z\in \mathbb{Z}^d: z\notin \Lambda, z \text{ has a nearest neighbor in }\Lambda \},
\end{equation}

\begin{equation}
\overline{\Lambda}:=\Lambda\cup \partial_{ex}\Lambda.
\end{equation}
We let $\mathscr{B}(\Lambda)$ be the set of all $\{z,w\}\in \mathbb{E}^d$ with $z,w\in \Lambda$, and $\overline{\mathscr{B}}(\Lambda)$ be the set of all $\{z,w\}$ with $z$ or $w\in \Lambda$. We will consider the extended graph $G=(V,E)$ where $V=\mathbb{Z}^d\cup\{g\}$ ($g$ is usually called the \textbf{ghost vertex} \cite{Gri67}) and $E$ is the set $\mathbb{E}^d \cup \{\{z,g\}:z\in \mathbb{Z}^d\}$. The edges in $\mathbb{E}^d$ are called \textbf{internal edges} while the edges in $\{\{z,g\}:z\in \mathbb{Z}^d\}$ are called \textbf{external edges}. We let $\mathscr{E}(\Lambda)$ be the set of all external edges with an endpoint in $\Lambda$, i.e.,
\[\mathscr{E}(\Lambda):=\left\{\left\{z,g\right\}:z\in \Lambda\right\}.\]

The random cluster model at inverse temperature $\beta$ on $\Lambda\subseteq \mathbb{Z}^d$ with boundary condition $\rho\in\{0,1\}^{\mathscr{B}(\Lambda^C)\cup\mathscr{E}(\Lambda^C)}$ and with external field $H\geq0$ is defined by the probability measure $\mathbb{P}_{\Lambda,\rho,\beta,H}$ on $\{0,1\}^{\overline{\mathscr{B}}(\Lambda)\cup\mathscr{E}(\Lambda)}$ such that for any $\omega\in\{0,1\}^{\overline{\mathscr{B}}(\Lambda)\cup\mathscr{E}(\Lambda)}$,
\begin{align}\label{eqRCdef}
\mathbb{P}_{\Lambda,\rho,\beta,H}(\omega)=\frac{q^{\mathcal{K}\left(\Lambda, (\omega\rho)_{\Lambda}\right)}}{\hat{Z}_{\Lambda,\rho,\beta,H}}\prod_{e\in\overline{\mathscr{B}}(\Lambda)}(1-e^{-2\beta})^{\omega_e}(e^{-2\beta})^{1-\omega_e}\prod_{e\in\mathscr{E}(\Lambda)}(1-e^{-2H})^{\omega_e}(e^{-2H})^{1-\omega_e},
\end{align}
where $(\omega\rho)_{\Lambda}$ denotes the configuration which coincides with $\omega$ on $\overline{\mathscr{B}}(\Lambda)\cup\mathscr{E}(\Lambda)$ and with $\rho$ on $\mathscr{B}(\Lambda^C)\cup\mathscr{E}(\Lambda^C)$, $\mathcal{K}\left(\Lambda, (\omega\rho)_{\Lambda}\right)$ denotes the number of clusters in $(\omega\rho)_{\Lambda}$ which intersect $\Lambda$ and do not contain $g$, and $\hat{Z}_{\Lambda,\rho,\beta,H}$ is the partition function. An edge $e$ is said to be \textbf{open} if $\omega_e=1$, otherwise it is said to be \textbf{closed}. $\mathbb{P}_{\Lambda,f,\beta, H}$ (respectively, $\mathbb{P}_{\Lambda,w,\beta,H}$) denotes the probability measure with free (respectively, wired) boundary conditions, i.e., $\rho\equiv 0$ (respectively, $\rho\equiv 1$) in \eqref{eqRCdef}.

The Potts models and random cluster models are related by the Edwards-Sokal coupling~\cite{ES88}; see also Sections 1.4 and 4.6 of \cite{Gri06}. Since we are mainly interested in the case $q=2$ (the Ising model) in this paper, we suppress the explicit reference to $q$ in this section.

For any $u,v\in \mathbb{Z}^d\cup\{g\}$, we write $u\longleftrightarrow v$ for the event that there is a path of open edges that connects $u$ and $v$, i.e., a path $u=z_0, z_1, \ldots, z_n=v$ with $e_i=\{z_i, z_{i+1}\}\in E$ and $\omega(e_i)=1$ for each $0\leq i<n$. For any $u,v\in \mathbb{Z}^d$, we write $u\overset{W}\longleftrightarrow v$ if $u\longleftrightarrow v$ and each vertex on this path is in $W\subseteq\mathbb{Z}^d\cup\{g\}$. For any $A, B\subseteq \mathbb{Z}^d\cup\{g\}$, we write $A\longleftrightarrow B$ if there is some $u\in A$ and $v\in B$ such that $u\longleftrightarrow v$. $A\centernot\longleftrightarrow B$ denotes the complement of $A\longleftrightarrow B$.

By the FKG lattice property,  $\mathbb{P}_{\Lambda,w,\beta, H}$ (with wired boundary conditions) has a unique infinite volume limit as $\Lambda\rightarrow \mathbb{Z}^d$, which we denote by $\mathbb{P}_{\mathbb{Z}^d,w,\beta,H}$. Let $\theta(\beta,H)$ be the percolation probability,
\begin{equation}
\theta(\beta,H):=\mathbb{P}_{\mathbb{Z}^d,w,\beta,H}(0\overset{\mathbb{Z}^d}{\longleftrightarrow}\infty),
\end{equation}
where $\{0\overset{\mathbb{Z}^d}{\longleftrightarrow}\infty\}$ is the event that the origin is in an infinite open cluster in $\mathbb{Z}^d$ (i.e., only using internal edges). The critical inverse temperature (with $H$=0) is defined by
\begin{equation}
\beta_c(d):=\sup\{\beta\geq0: \theta(\beta,0)=0\}.
\end{equation}
The Kert\'{e}sz line (see \cite{Ker89,BGLRS08,RW08}) is the function
\begin{equation}
H_K(\beta):=\sup\{H\geq 0: \theta(\beta,H)=0\}.
\end{equation}
Note that $\mathbb{P}_{\mathbb{Z}^d,w,\beta,H}$ is stochastically increasing in $\beta$ and $H$ (i.e., in the FKG sense), so $H_K(\beta)$ is decreasing in $\beta$. It is clear that $H_K(\beta)=0$ for each $\beta>\beta_c(d)$. It is also easy to see that $H_K(\beta)=\infty$ if $\beta< \beta_P(d)$ where
\begin{equation}\label{eqbetaPdef}
1-e^{-2\beta_P(d)}=p_c^b(d)
\end{equation}
and $p_c^b(d)$ is the critical probability for independent Bernoulli bond percolation on $\mathbb{Z}^d$. It follows from Theorem 1.4 of \cite{CS00} and the Edwards-Sokal coupling that $H_K(\beta_P(d))=\infty$ and $H_K(\beta)<\infty$ for each $\beta>\beta_P(d)$; to see this, note that the percolation of state~$1$ in the $q$-state Potts model with $\beta\in[0,\infty)$ and an external field $H\in[0,\infty)$ applied to state~$1$ is stochastically dominated by (respectively, dominates) an independent Bernoulli site percolation on $\mathbb{Z}^d$ with $p<1$ (respectively, with $\tilde{p}=\tilde{p}(H)<1$ where $\tilde{p}(H)\rightarrow1$ as $H\rightarrow\infty$). A similar argument also implies that $\beta_P(d)<\beta_c(d)$ (see also Theorem 3.1 of \cite{Gri95}). Using arguments similar to those developed in \cite{BGK93}, we will show the following theorem about the Kert\'{e}sz line.

\begin{theorem}\label{thmKer}
Suppose $d\geq 1$ and $q\in \{2,3,\dots\}$. We have
\begin{equation}
H_K(\beta_c(d))=0.
\end{equation}
Moreover, there exists $\epsilon_d\in(0,\beta_c(d)-\beta_P(d))$ such that
\begin{equation}\label{eqHKrange}
H_K(\beta_P(d)+\epsilon_d)\in(0,\infty).
\end{equation}
\end{theorem}
Recall that $\Lambda_n:=[-n,n]^d$. Let $\vec{p}=(p_1,p_2):=(1-e^{-2\beta},1-e^{-2H})$.  Define
\begin{equation}
\theta_n(\vec{p}):=\mathbb{P}_{\Lambda_n,w,\beta,H}(0 \overset{\mathbb{Z}^d}{\longleftrightarrow} \partial_{in}\Lambda_n).
\end{equation}
In order to prove Theorem \ref{thmKer}, we will use the following two propositions.
\begin{proposition}\label{propcomp1}
Suppose $d\geq 1$ and $q>1$. There exist continuous functions $\alpha,\gamma:(0,1)^2\rightarrow(0,\infty)$ and $N>0$ such that
\begin{equation}
\alpha(\vec{p})\frac{\partial \theta_n}{\partial p_1}\leq \frac{\partial \theta_n}{\partial p_2}\leq \gamma(\vec{p})\frac{\partial \theta_n}{\partial p_1}
\end{equation}
for all $\vec{p}\in(0,1)^2$, all $n\geq N$.
\end{proposition}
\begin{proof}
The proof is similar to that of Theorem 1 in \cite{BGK93}; here we only explain the changes needed for our setting. Let $\mathbf{p}:=(p_e,e\in\mathscr{B}(\Lambda_n)\cup\mathscr{E}(\Lambda_n))$. Let $\tilde{\mathbb{P}}_{\Lambda_n,\bar{w},\mathbf{p}}$ be the generalized random cluster measure, i.e., with $e^{-2\beta}$ and $e^{-2H}$ replaced by $1-p_e$ in \eqref{eqRCdef}. For each $e=\{u,v\}\in\overline{\mathscr{B}}(\Lambda_n)$ with either $u$ or $v\in \Lambda_n$, let $f:=\{u,g\}$ if $u\in \Lambda_n$ and set $f:=\{v,g\}$ otherwise. Our goal is to show
\begin{equation}
\frac{\partial \theta_n}{\partial p_e}\leq \tilde{\alpha}(\mathbf{p})\frac{\partial \theta_n}{\partial p_f}
\end{equation}
for some continuous function $\tilde{\alpha}:(0,1)^{\overline{\mathscr{B}}(\Lambda_n)\cup\mathscr{E}(\Lambda_n)}\rightarrow (0,\infty)$. Let
\begin{equation}
\langle B_u\rangle:=\big(\mathscr{B}(u+\Lambda_2)\cap\mathscr{B}(\Lambda_n)\big)\cup\big(\mathscr{E}(u+\Lambda_2)\cap\mathscr{E}(\Lambda_n)\big).
\end{equation}
We modify the events $V^i, i=1,2,3$ in \cite{BGK93}, as follows. Let $f^{\prime}:=\{v,g\}$ if $u\in \Lambda_n$ and $f^{\prime}:=\{u,g\}$ otherwise.
\begin{enumerate}[(i)]
\item $V^1$ is the event that during the time-interval $(t,t+1]$, all edges in $\langle B_u\rangle$ which are present  in $X_t$ are removed, and no edges in $\langle B_u\rangle$ are added to $X$; $e$ remains present in Y.

\item $V^2$ is the event that during $(t+1,t+2]$, the edges $f$ and $f^{\prime}$ are added to $X$, but no other edges in $\langle B_u\rangle$ are added to $X$; e remains present in Y.

\item $V^3$ is the event that during $(t+2,t+3]$, the edge $f$ is removed from $X$ but not from~$Y$.
\end{enumerate}
The rest of the proof is the same as that of Theorem 1 in \cite{BGK93}.
\end{proof}

The same proof as in Theorem 2 of \cite{BGK93} yields the following proposition. Note that that theorem holds for each $\mathbf{p}\in(0,1)^K$, as defined in \cite{BGK93}. Let $U:=\{\vec{e}\in\mathbb{R}^2:|\vec{e}|=1\}$. An open subset $V$ of $U$ is called \textbf{full} if $\{(x_1,x_2)\in U: x_1>0 \text{ or }x_2>0\}\subseteq V$.

\begin{proposition}\label{propcomp2}
Suppose $d\geq 1$ and $q>1$. For any $\vec{p}\in (0,1)^2$, there exist $c_1,c_2,\epsilon_0\in(0,\infty)$, and a full subset $V$ of $U$ such that
\begin{equation}
\theta(\vec{p}+c_1\epsilon \vec{e})\leq \theta(\vec{p}+\epsilon\vec{f})\leq \theta(\vec{p}+c_2\epsilon \vec{e})
\end{equation}
for all $0<\epsilon<\epsilon_0$ and $\vec{e},\vec{f}\in V$.
\end{proposition}

Now we are ready to prove Theorem \ref{thmKer}.
\begin{proof}[Proof of Theorem \ref{thmKer}]
We prove the theorem by contradiction. Suppose $H_K(\beta_c(d))>0$. Then for any $H_0\in[0,H_K(\beta_c(d)))$ we have $\theta(\beta_c(d),H_0)=0$. Proposition \ref{propcomp2} implies that
\begin{equation}
\theta(\beta_c(d)+\epsilon_1, \tilde{H}_0)=0 \text{ for some }\epsilon_1, \tilde{H}_0>0,
\end{equation}
which contradicts the fact that $\theta(\beta_c(d)+\epsilon_1,0)>0$. The proof of the second part of the theorem is a similar argument by contradiction, as follows. We assume that
\begin{equation}
H_K(\beta_P(d)+\epsilon)=0 \text{ for all }\epsilon>0.
\end{equation}
That is
\begin{equation}
\theta(\beta_P(d)+\epsilon,H)>0 \text{ for all }\epsilon>0, H>0.
\end{equation}
Then one gets a contradiction by Proposition \ref{propcomp2} and $H_K(\beta_P(d))=\infty$. Finally, we note that the assertion of \eqref{eqHKrange} that  $H_K(\beta_P(d)+\epsilon_d)<\infty$ follows from the argument given above after \eqref{eqbetaPdef}.
\end{proof}

We summarize our results about the Kert\'{e}sz line in the following theorem (which includes the results of Theorem \ref{thmKer}).
\begin{theorem}\label{thmKerIsing}
\begin{enumerate}[(i)]
\item
For $d= 2$ and $q=2$,  we have
\begin{equation}
 H_K(\beta) \text{ is }\begin{cases}
      \infty, & \beta\in[0,\beta_P(2)], \\
      \in(0,\infty), &\beta\in(\beta_P(2),\beta_c(2)),\\
      0, & \beta\geq \beta_c(2).
   \end{cases}
\end{equation}
\item
For $d\geq 1$ and $q\in \{2,3,\dots\}$,  there exists $\epsilon_d>0$ such that
\begin{equation}
 H_K(\beta)\text{ is }\begin{cases}
      \infty, & \beta\in[0,\beta_P(d)],\\
      \in(0,\infty), & \beta\in (\beta_P(d),\beta_P(d)+\epsilon_d],\\
      \in[0,\infty), & \beta \in (\beta_P(d)+\epsilon_d, \beta_c(d)),\\
      0, & \beta\geq \beta_c(d).
   \end{cases}
\end{equation}

\end{enumerate}
\end{theorem}
\begin{remark}
It is natural to conjecture that $H_K(\beta)>0$ if $\beta \in (\beta_P(d), \beta_c(d))$.
\end{remark}
\begin{proof}[Proof of Theorem \ref{thmKerIsing}]
By Theorem 1 in \cite{Hig93}, for each $\beta\in [0,\beta_c(2)) $ there exists $H_{Hig}(\beta)>0$ such that there is no infinite $+$ cluster for the Ising model on $\mathbb{Z}^2$ at inverse temperature $\beta$ with external field $H\in[0,H_{Hig}(\beta))$. By the Edwards-Sokal coupling, the existence of an infinite open cluster in the random cluster model implies the existence of either an infinite $+$ or $-$ cluster in the corresponding Ising model. Therefore, $\theta(\beta,H)=0$ for any $\beta\in [0,\beta_c(2)) $ and $H\in[0,H_{Hig}(\beta))$. This, Theorem \ref{thmKer} and the argument near \eqref{eqbetaPdef} complete the proof of the first part of the theorem. The second part of the theorem follows from Theorem \ref{thmKer} and the argument near \eqref{eqbetaPdef}.
\end{proof}

Our next result is the exponential decay of $\mathbb{P}_{\Lambda,w,\beta, H}$ for $(\beta,H)$ under or on the left of the Kert\'{e}sz line. More precisely,
\begin{theorem}\label{thmPCdecay}
Suppose $d\geq 1$ and $q>1$. For any $\beta \in [0, \beta_c(d))$ and $H\in [0,H_K(\beta))$, there exists $C_{8}(\beta,H)\in(0,\infty)$ such that for each $L\geq  0$,
\begin{equation}
\mathbb{P}_{\Lambda_L, w,\beta,H}(0\overset{\mathbb{Z}^d}{\longleftrightarrow} \partial_{in}\Lambda_L)\leq e^{-C_{8}(\beta,H)L}.
\end{equation}
\end{theorem}
\begin{proof}
Note that $\mathbb{P}_{\Lambda,w,\beta, H}$ has the FKG lattice property and the domain Markov property. So the proof of Theorem 1.2 in \cite{DCRT17} applies to  $\mathbb{P}_{\Lambda_L,w,\beta, H}$ (and $\mathbb{P}_{\mathbb{Z}^d,w,\beta,H}$).
\end{proof}

\section{Weak mixing property}\label{secMixing}
In this section, we prove Theorems \ref{thmSBSDB}-\ref{thmABG}.
\subsection{Generalized random cluster model}
We first generalize the random cluster model with constant external field in \eqref{eqRCdef} to a more general external field. Let $\Lambda\subseteq \mathbb{Z}^d$ be finite and $\vec{h}\in \mathbb{R}^{\Lambda}$. Consider the extended graph $(V,E)$ where $V=\Lambda\cup\{g^{+},g^{-}\}$, where $g^{+}$ and $g^{-}$ represent the $+$ and $-$ ghosts respectively and $E=\mathbb{E}^d\cup\mathcal{E}^{\vec{h}}_{+}(\Lambda)\cup\mathcal{E}^{\vec{h}}_{-}(\Lambda)$ with
\begin{equation}
\mathcal{E}^{\vec{h}}_{+}(\Lambda):=\{\{z,g^{+}\}:z\in\Lambda \text{ and }h_z>0\},
\end{equation}
\begin{equation}
\mathcal{E}^{\vec{h}}_{-}(\Lambda):=\{\{z,g^{-}\}:z\in\Lambda \text{ and }h_z<0\}.
\end{equation}

Let $\mathscr{E}_{+}(\Lambda^C)\cup\mathscr{E}_{-}(\Lambda^C)$ be the set of external edges with one endpoint in $\Lambda^C$ and the other endpoint in $\{g^{+},g^{-}\}$ (whether it is $g^{+}$ or $g^{-}$ will be clear from the context). The random cluster model at $\beta$ on $\Lambda\subseteq \mathbb{Z}^d$ with boundary condition $\rho\in\{0,1\}^{\mathscr{B}(\Lambda^C)\cup\mathscr{E}_{+}(\Lambda^C)\cup\mathscr{E}_{-}(\Lambda^C)}$ and with external field $H\vec{h}\in\mathbb{R}^{\Lambda}$ is defined by the probability measure $\mathbb{P}_{\Lambda,\rho,\beta,H}^{\vec{h}}$ on $\{0,1\}^{\overline{\mathscr{B}}(\Lambda)\cup\mathcal{E}^{\vec{h}}_{+}(\Lambda)\cup\mathcal{E}^{\vec{h}}_{-}(\Lambda)}$ such that for any $\omega\in\{0,1\}^{\overline{\mathscr{B}}(\Lambda)\cup\mathcal{E}^{\vec{h}}_{+}(\Lambda)\cup\mathcal{E}^{\vec{h}}_{-}(\Lambda)}$,
\begin{align}\label{eqRCGdef}
\mathbb{P}_{\Lambda,\rho,\beta,H}^{\vec{h}}(\omega)=&\frac{q^{\mathcal{K}\left(\Lambda, (\omega\rho)_{\Lambda}\right)}}{\hat{Z}_{\Lambda,\rho,\beta,H}^{\vec{h}}}1_{\{g^{+}\centernot\longleftrightarrow g^{-}\}}(\omega)\prod_{e\in\overline{\mathscr{B}}(\Lambda)}(1-e^{-2\beta})^{\omega_e}(e^{-2\beta})^{1-\omega_e}\nonumber\\
&\times\prod_{e\in\mathcal{E}^{\vec{h}}_{+}(\Lambda)\cup\mathcal{E}^{\vec{h}}_{-}(\Lambda)}(1-e^{-2H|h_e|})^{\omega_e}(e^{-2H|h_e|})^{1-\omega_e},
\end{align}
where $\mathcal{K}\left(\Lambda, (\omega\rho)_{\Lambda}\right)$ denotes the number of clusters in $(\omega\rho)_{\Lambda}$ which intersect $\Lambda$ and contain neither $g^{+}$ nor $g^{-}$, $h_e:=h_z$ for each external edge $e=\{z,g^{+}\}$ or $\{z,g^{-}\}$, $1_{\{\cdot\}}$ denotes the indicator function, and other notation is similar to that in \eqref{eqRCdef}. We now define $\mathbb{P}_{\Lambda,\rho,\beta,H}^{|\vec{h}|}$ on the same graph where $\mathbb{P}_{\Lambda,\rho,\beta,H}^{\vec{h}}$ lives using the right hand side of  \eqref{eqRCGdef} except with the indicator function removed. We note that $\mathbb{P}_{\Lambda,\rho,\beta,H}^{|\vec{h}|}$ is not quite the same as replacing each $h_z$ by $|h_z|$ in $\mathbb{P}_{\Lambda,\rho,\beta,H}^{\vec{h}}$ but these two measures have the same marginal on $\overline{\mathscr{B}}(\Lambda)$ because our definition of $\mathcal{K}(\Lambda,\cdot)$ implies that $g^+$ and $g^-$ are wired in $\mathbb{P}_{\Lambda,\rho,\beta,H}^{|\vec{h}|}$. An easy observation is that $\mathbb{P}_{\Lambda,\rho,\beta,H}^{\vec{h}}$ is stochastically dominated by $\mathbb{P}_{\Lambda,\rho,\beta,H}^{|\vec{h}|}$.
\begin{lemma}\label{lemsd}
For any increasing event $A\subseteq \{0,1\}^{\overline{\mathscr{B}}(\Lambda)\cup\mathcal{E}^{\vec{h}}_{+}(\Lambda)\cup\mathcal{E}^{\vec{h}}_{-}(\Lambda)}$, we have
\begin{equation}
\mathbb{P}_{\Lambda,\rho,\beta,H}^{\vec{h}}(A)\leq\mathbb{P}_{\Lambda,\rho,\beta,H}^{|\vec{h}|}(A).
\end{equation}
\end{lemma}
\begin{proof}
Note that the Radon-Nikodym derivative
\begin{equation}
\frac{d\mathbb{P}_{\Lambda,\rho,\beta,H}^{\vec{h}}}{d\mathbb{P}_{\Lambda,\rho,\beta,H}^{|\vec{h}|}}(\omega)=C(\Lambda,\rho,\beta,H,\vec{h})1_{\{g^{+}\centernot\longleftrightarrow g^{-}\}}(\omega)
\end{equation}
is a decreasing function (in the FKG sense) where $C(\Lambda,\rho,\beta,H,\vec{h})$ is a constant that only depends on $\Lambda,\rho,\beta,H, \vec{h}$. Since $\mathbb{P}_{\Lambda,\rho,\beta,H}^{|\vec{h}|}$ satisfies the FKG inequality, the conclusion of the lemma follows.
\end{proof}

Next, we bound the total variation of $\mathbb{P}^{\vec{h}}_{\Lambda,\rho,\beta,H}(\omega_{\Delta}\in \cdot)$ and $\mathbb{P}^{\vec{h}}_{\Lambda,\rho^{\prime},\beta,H}(\omega_{\Delta}\in \cdot)$ by a connectivity probability.
\begin{lemma}\label{lemcoupleRC}
Let $\Delta\subseteq \Lambda$ be finite subsets of $\mathbb{Z}^d$. Then
\begin{align}
&\quad\sup_{\rho,\rho^{\prime}\in\{0,1\}^{\mathscr{B}(\Lambda^C)\cup\mathscr{E}_{+}(\Lambda^C)\cup\mathscr{E}_{-}(\Lambda^C)}}TV\left(\mathbb{P}^{\vec{h}}_{\Lambda,\rho,\beta,H}(\omega_{\Delta}\in \cdot), \mathbb{P}^{\vec{h}}_{\Lambda,\rho^{\prime},\beta,H}(\omega_{\Delta}\in \cdot)\right)\nonumber\\
&\leq\mathbb{P}_{\Lambda,w,\beta,H}^{|\vec{h}|}(\partial_{in}\Delta \overset{\mathbb{Z}^d}{\longleftrightarrow} \partial_{ex}\Lambda),
\end{align}
where $w$ in the subscript denotes the wired boundary condition.
\end{lemma}
\begin{proof}
The proof uses couplings similar to those in the proofs of Theorem 3.11 of \cite{New97} and Lemma 2.3 of \cite{Ale98} and in the discussion on p. 827 of \cite{LS12}. For completeness, we spell out the details here. Note that $\mathbb{P}^{|\vec{h}|}_{\Lambda,w,\beta,H}$ stochastically dominates  $\mathbb{P}_{\Lambda,\rho,\beta,H}^{|\vec{h}|}$ and $\mathbb{P}^{|\vec{h}|}_{\Lambda,\rho^{\prime},\beta,H}$, and thus also $\mathbb{P}^{\vec{h}}_{\Lambda,\rho,\beta,H}$ and $\mathbb{P}^{\vec{h}}_{\Lambda,\rho^{\prime},\beta,H}$ by Lemma \ref{lemsd}. Our goal is to find a coupling of $\mathbb{P}^{\vec{h}}_{\Lambda,\rho,\beta,H}$, $\mathbb{P}^{\vec{h}}_{\Lambda,\rho^{\prime},\beta,H}$ and $\mathbb{P}^{|\vec{h}|}_{\Lambda,w,\beta,H}$, such that the three configurations from those three probability measures coincide inside the boundary cluster of $\mathbb{P}^{|\vec{h}|}_{\Lambda,w,\beta,H}$.

We proceed as follows. Order the set of edges in $\overline{\mathscr{B}}(\Lambda)\cup\mathscr{E}^{\vec{h}}_{+}(\Lambda)\cup\mathscr{E}^{\vec{h}}_{-}(\Lambda)$ in such a way that $e_1$ precedes $e_2$ in the ordering if $d(\tilde{e}_1,\Lambda^C)<d(\tilde{e}_2,\Lambda^C)$ where $\tilde{e}_1$ is the set $\{x,y\}$ if $e_1=\{x,y\}$ with $x,y\in\mathbb{Z}^d$ and $\tilde{e}_1:=\{z\}$ if $e_1=\{z,g^{+}\}$ or $\{z,g^{-}\}$ with $z\in\mathbb{Z}^d$. We progressively explore in unit time steps the boundary open cluster of the configuration $\omega^w$ in the support of $\mathbb{P}_{\Lambda,w,\beta,H}^{|\vec{h}|}$. Denote by $G_t$ the set of edges revealed up to (and including) the integer time $t$, and let $E_t$ be the set of open edges in $G_t\cap\overline{\mathscr{B}}(\Lambda)$, where we use the following process of revealing edges.
\begin{itemize}
\item $G_0=\emptyset$ and $E_0=\emptyset$.

\item For each $t\geq 0$, reveal the smallest (according to the above ordering) unexplored edge $e$ that is adjacent to $E_t\cup\mathscr{B}(\Lambda^C)$, setting its value in $\omega^w$ via an independent $[0,1]$ uniformly distributed random variable $U_e$:
\begin{equation}
\omega^w_e:=1\{U_e\leq \mathbb{P}_{\Lambda,w,\beta,H}^{|\vec{h}|}(\omega_e=1|\omega_{G_t}=\omega^w_{G_t})\}.
\end{equation}
We also explore the configurations in the other two measures as follows:
\begin{eqnarray}
&\omega^{\rho}_e:=1\{U_e\leq \mathbb{P}_{\Lambda,\rho,\beta,H}^{\vec{h}}(\omega_e=1|\omega_{G_t}=\omega^{\rho}_{G_t})\},\\
&\omega^{\rho^{\prime}}_e:=1\{U_e\leq \mathbb{P}_{\Lambda,\rho^{\prime},\beta,H}^{\vec{h}}(\omega_e=1|\omega_{G_t}=\omega^{\rho^{\prime}}_{G_t})\}.
\end{eqnarray}
Let $G_{t+1}:=G_t\cup\{e\}$ and let
\begin{equation}
E_{t+1}=\begin{cases}
      E_t\cup\{e\}, & e\in\overline{\mathscr{B}}(\Lambda) \text{ and } \omega^w_e=1,\\
      E_t, & \text{otherwise.}
   \end{cases}
\end{equation}

\item Let $\tau$ be the first time $t$ at which there is no unexplored edge $e\in\overline{\mathscr{B}}(\Lambda)$ that is adjacent to $E_t\cup\mathscr{B}(\Lambda^C)$.
\end{itemize}
One may show by induction that $\omega^{\rho}_{G_t}\leq \omega^w_{G_t}$ and $\omega^{\rho^{\prime}}_{G_t}\leq \omega^w_{G_t}$ for each $t\in[0,\tau]$ (we only show $\omega^{\rho}_{G_t}\leq \omega^w_{G_t}$  since the other proof is the same):
Suppose $\omega^{\rho}_{G_t}\leq \omega^w_{G_t}$  and let $e$ be the edge that will be explored at time $t+1$; then by Lemma \ref{lemsd},
\begin{equation}
\mathbb{P}_{\Lambda,\rho,\beta,H}^{\vec{h}}(\omega_e=1|\omega_{G_t}=\omega^{\rho}_{G_t})\leq  \mathbb{P}_{\Lambda,\rho,\beta,H}^{|\vec{h}|}(\omega_e=1|\omega_{G_t}=\omega^{\rho}_{G_t}).
\end{equation}
By the induction assumption and the FKG lattice property for $ \mathbb{P}_{\Lambda,\cdot,\beta,H}^{|\vec{h}|}$,
\begin{align}
\mathbb{P}_{\Lambda,\rho,\beta,H}^{|\vec{h}|}(\omega_e=1|\omega_{G_t}=\omega^{\rho}_{G_t}))&\leq  \mathbb{P}_{\Lambda,w,\beta,H}^{|\vec{h}|}(\omega_e=1|\omega_{G_t}=\omega^{\rho}_{G_t})\nonumber\\
&\leq  \mathbb{P}_{\Lambda,w,\beta,H}^{|\vec{h}|}(\omega_e=1|\omega_{G_t}=\omega^{w}_{G_t}).
\end{align}

After time $\tau$, we may reveal the remaining edges according to their ordering while following the same process described above. It is easy to see that $\omega^{\rho}_e=\omega^{\rho^{\prime}}_e$ for those remaining edges since the closed edges in $\overline{\mathscr{B}}(\Lambda)$ adjacent to $E_{\tau}$ serve as the common boundary conditions for $\mathbb{P}^{\vec{h}}_{\Lambda,\rho,\beta,H}$ and $\mathbb{P}^{\vec{h}}_{\Lambda,\rho^{\prime},\beta,H}$. Therefore, under this coupling
\begin{equation}
\omega_{\Delta}^{\rho}\neq\omega_{\Delta}^{\rho^{\prime}} \text{ if and only if }\partial_{in}\Delta \overset{\mathbb{Z}^d}{\longleftrightarrow} \partial_{ex}\Lambda \text{ in } \omega^w,
\end{equation}
which completes the proof of the lemma.
\end{proof}

\subsection{Edwards-Sokal coupling for the RFIM}\label{subsecES}
The Edwards-Sokal coupling \cite{ES88} is a probability measure on a common probability space for the random cluster and Ising models. We will restrict attention to the case where there is a boundary condition $\eta\in\{-1,+1\}^{\partial_{ex} \Lambda}$, only for the Ising variables. More precisely, let $\hat{\mathbb{P}}^{\vec{h}}_{\Lambda,\eta,\beta,H}$ be the following probability measure on $\{-1,+1\}^{\Lambda}\times\{0,1\}^{\overline{\mathscr{B}}(\Lambda)\cup\mathcal{E}^{\vec{h}}_{+}(\Lambda)\cup\mathcal{E}^{\vec{h}}_{-}(\Lambda)}$. For each $(\sigma,\omega)\in\{-1,+1\}^{\Lambda}\times\{0,1\}^{\overline{\mathscr{B}}(\Lambda)\cup\mathcal{E}^{\vec{h}}_{+}(\Lambda)\cup\mathcal{E}^{\vec{h}}_{-}(\Lambda)}$,
\begin{align}\label{eqCoupledef}
\hat{\mathbb{P}}_{\Lambda,\eta,\beta,H}^{\vec{h}}(\sigma,\omega)\propto&~1_{F(\Lambda,\eta)}(\sigma,\omega)\prod_{e\in\overline{\mathscr{B}}(\Lambda)}(1-e^{-2\beta})^{\omega_e}(e^{-2\beta})^{1-\omega_e}\nonumber\\
&\times\prod_{e\in\mathcal{E}^{\vec{h}}_{+}(\Lambda)\cup\mathcal{E}^{\vec{h}}_{-}(\Lambda)}(1-e^{-2H|h_e|})^{\omega_e}(e^{-2H|h_e|})^{1-\omega_e}.
\end{align}
Here, $F(\Lambda,\eta)$ is the event
\[\{(\sigma,\omega)\in\{-1,+1\}^{\Lambda}\times\{0,1\}^{\overline{\mathscr{B}}(\Lambda)\cup\mathcal{E}^{\vec{h}}_{+}(\Lambda)\cup\mathcal{E}^{\vec{h}}_{-}(\Lambda)}: (\sigma\eta)_{\Lambda}(x)=(\sigma\eta)_{\Lambda}(y)  \text{ for each }\omega_{\{x,y\}}=1\}\]
(where $(\sigma\eta)_{\Lambda}$ is the configuration which coincides with $\sigma$ on $\Lambda$ and with $\eta$ on $\partial_{ex}\Lambda$, and we assign $+1$ to $g^{+}$ and $-1$ to $g^{-}$); the constant of proportionality is chosen so that
\begin{equation}
\sum_{(\sigma,\omega)}\hat{\mathbb{P}}_{\Lambda,\rho,\beta,H}^{\vec{h}}(\sigma,\omega)=1.
\end{equation}

The marginal of $\hat{\mathbb{P}}_{\Lambda,\eta,\beta,H}^{\vec{h}}$ on $\{-1,+1\}^{\Lambda}$ is $P_{\Lambda,\eta,\beta,H}^{\vec{h}}$. The marginal on $\{0,1\}^{\overline{\mathscr{B}}(\Lambda)\cup\mathcal{E}^{\vec{h}}_{+}(\Lambda)\cup\mathcal{E}^{\vec{h}}_{-}(\Lambda)}$ is $\mathbb{P}_{\Lambda,\eta,\beta,H}^{\vec{h}}(\cdot):=\mathbb{P}_{\Lambda,w,\beta,H}^{|\vec{h}|}(\cdot | E(\Lambda,\eta))$, where one may interpret $w$ as putting an open external edge between each vertex in $\{x\in\partial_{ex}\Lambda:\eta_x=+1\}$ (respectively $\{x\in\partial_{ex}\Lambda:\eta_x=-1\}$) and $g^{+}$ (respectively $g^{-}$), and all other edges in $\mathscr{B}(\Lambda^C)$ are closed; $E(\Lambda,\eta)$ is the event
\begin{equation*}
\{\omega\in\{0,1\}^{\overline{\mathscr{B}}(\Lambda)\cup\mathcal{E}^{\vec{h}}_{+}(\Lambda)\cup\mathcal{E}^{\vec{h}}_{-}(\Lambda)}: x\overset{\bar{\Lambda}\cup\{g^{+},g^{-}\}}{\centernot\longleftrightarrow} y \text{ for any }x,y \in \partial_{ex}\Lambda\cup\{g^{+},g^{-}\} \text{ with }\eta_x\neq\eta_y\}.
\end{equation*}
The following lemma is obvious but will be useful.
\begin{lemma}
For any increasing event  $A\subseteq \{0,1\}^{\overline{\mathscr{B}}(\Lambda)\cup\mathcal{E}^{\vec{h}}_{+}(\Lambda)\cup\mathcal{E}^{\vec{h}}_{-}(\Lambda)}$, we have
\begin{equation}
\mathbb{P}_{\Lambda,\eta,\beta,H}^{\vec{h}}(A)=\mathbb{P}_{\Lambda,w,\beta,H}^{|\vec{h}|}(A | E(\Lambda,\eta))\leq \mathbb{P}_{\Lambda,w,\beta,H}^{|\vec{h}|}(A).
\end{equation}
\end{lemma}
\begin{proof}
This follows from the FKG inequality for $\mathbb{P}_{\Lambda,w,\beta,H}^{|\vec{h}|}$ (note that $E(\Lambda,\eta)$ is a decreasing event).
\end{proof}

For $\omega\in \{0,1\}^{\overline{\mathscr{B}}(\Lambda)\cup\mathcal{E}^{\vec{h}}_{+}(\Lambda)\cup\mathcal{E}^{\vec{h}}_{-}(\Lambda)}$, the conditional measure $\hat{\mathbb{P}}_{\Lambda,\eta,\beta,H}^{\vec{h}}(\cdot|\omega)$ on $\{-1,+1\}^{\Lambda}$ is realized by tossing independent fair coins --- one for each open cluster not containing $g^{+}$ or $g^{-}$ --- and then setting $\sigma_x=+1$ for all vertices $x$ in a cluster with heads and $-1$ for tails. For $x$ in the cluster of $g^{+}$ (respectively, $g^{-}$), $\sigma_x=+1$ (respectively, $\sigma_x=-1$).

\subsection{Weak mixing property for a small field}
We first consider $\beta\in[0,\beta_c(d))$ and $H$ small. The following coupling between RFIMs with different boundary conditions is very important.
\begin{lemma}\label{lemcoupleIsing}
Let $\Delta\subseteq \Lambda$ be finite subsets of $\mathbb{Z}^d$. Then
\begin{equation}
\sup_{\eta,\eta^{\prime}\in\{-1,+1\}^{\partial_{ex} \Lambda}}TV\left(P^{\vec{h}}_{\Lambda,\eta,\beta,H}(\sigma_{\Delta}\in \cdot), P^{\vec{h}}_{\Lambda,\eta^{\prime},\beta,H}(\sigma_{\Delta}\in \cdot)\right)\leq\mathbb{P}_{\Lambda,w,\beta,H}^{|\vec{h}|}(\partial_{in}\Delta\overset{\mathbb{Z}^d}{\longleftrightarrow} \partial_{ex}\Lambda).
\end{equation}
\end{lemma}
\begin{proof}
Our argument is similar to that used for Lemma 6.2 in \cite{Ale98}. A proof like that of Lemma \ref{lemcoupleRC} above gives a coupling of the three measures $\mathbb{P}_{\Lambda,\eta,\beta,H}^{\vec{h}}$, $\mathbb{P}_{\Lambda,\eta^{\prime},\beta,H}^{\vec{h}}$ and $\mathbb{P}_{\Lambda,+,\beta,H}^{|\vec{h}|}$ (the last of these three is the same as  $\mathbb{P}_{\Lambda,w,\beta,H}^{|\vec{h}|}$) such that $\omega^{\eta}$ from $\mathbb{P}_{\Lambda,\eta,\beta,H}^{\vec{h}}$ and $\omega^{\eta^{\prime}}$ from $\mathbb{P}_{\Lambda,\eta^{\prime},\beta,H}^{\vec{h}}$ agree on $\overline{\mathscr{B}}(\Lambda\setminus C^+_{\partial_{ex}\Lambda})\cup \mathscr{E}(\Lambda\setminus C^+_{\partial_{ex}\Lambda})$ where
\begin{equation}
C^+_{\partial_{ex}\Lambda}:=\{x\in\Lambda: x\overset{\mathbb{Z}^d}{\longleftrightarrow}\partial_{ex}\Lambda \text{ in }\omega^+ \text{ from } \mathbb{P}_{\Lambda,+,\beta,H}^{|\vec{h}|}\}.
\end{equation}
To get a configuration $\sigma^{\eta}$ from $P^{\vec{h}}_{\Lambda,\eta,\beta,H}$, by the Edwards-Sokal coupling in Subsection \ref{subsecES}, one assigns $+1$ or $-1$ with equal probability to each open cluster of $\omega^{\eta}$ disjoint from $C^+_{\partial_{ex}\Lambda}$ and from the clusters of $g^+,g^-$. This assignment can be done identically for the clusters of  $\omega^{\eta}$ and $\omega^{\eta^{\prime}}$ disjoint from  $C^+_{\partial_{ex}\Lambda}$, which yields a coupling of $P^{\vec{h}}_{\Lambda,\eta,\beta,H}$ and $P^{\vec{h}}_{\Lambda,\eta^{\prime},\beta,H}$ such that the corresponding configurations $(\sigma^{\eta},\sigma^{\eta^{\prime}})$ agree on $\Lambda \setminus C^+_{\partial_{ex}\Lambda}$. This completes the proof of the lemma.
\end{proof}

We are ready to prove Theorems \ref{thmSBSDB}-\ref{thmSBG}.

\begin{proof}[Proof of Theorems \ref{thmSBSDB} and \ref{thmSBLDB}]
In this case, we have $|h_x|=1$ for each $x\in\mathbb{Z}^d$. So the marginal on $\overline{\mathscr{B}}(\Lambda)$ of $\mathbb{P}_{\Lambda,w,\beta,H}^{|\vec{h}|}$ is the same as that of $\mathbb{P}_{\Lambda,w,\beta,H}$ which is defined in \eqref{eqRCdef}. Hence these two theorems follow from Lemma \ref{lemcoupleIsing} and Theorems \ref{thmKerIsing} and \ref{thmPCdecay}.
\end{proof}
\begin{proof}[Proof of Theorem \ref{thmSBG}]
In this case, $\mathbb{P}_{\Lambda,w,\beta,H}^{|\vec{h}|}$ is stochastically dominated by $\mathbb{P}_{\Lambda,w,\beta,\infty}$. Note that when $\beta\in[0,\beta_P(d))$, $\mathbb{P}_{\Lambda,w,\beta,\infty}$ is an independent Bernoulli bond percolation with probability $1-e^{-2\beta}<p_c^b(d)$, which is subcritical. So the theorem follows from Lemma~\ref{lemcoupleIsing}.
\end{proof}

\subsection{Weak mixing property for a large field}
In this subsection, we consider the case when $H$ is large. The idea is to find a set of vertices (such as a *-circuit when $d=2$) in the annulus $\Lambda\setminus\Delta$ which serves as the location of a common boundary condition for $P_{\Lambda,\eta,\beta,H}^{\vec{h}}$ and $P_{\Lambda,\eta^{\prime},\beta,H}^{\vec{h}}$. The proof is somewhat similar to that of Lemma \ref{lemcoupleRC}. Here are the details.

We order the vertices of $\Lambda=\{x_1,x_2,\dots\}$ in such a way that $x$ precedes $y$ in the ordering if $d(x,\Lambda^C)<d(y,\Lambda^C)$.  We explore vertices of the configuration $\sigma^{\eta}$ from $P^{\vec{h}}_{\Lambda,\eta,\beta,H}$ and $\sigma^{\eta^{\prime}}$ from $P^{\vec{h}}_{\Lambda,\eta^{\prime},\beta,H}$ that are connected by an open path  to $\partial_{ex}\Lambda$ of a certain site percolation process (that we are about to construct). We will denote this site percolation on $\Lambda$ by $\mathbf{S}=\{S_x:x\in\Lambda\}$ with the boundary condition defined by
\begin{equation}
S_x=1\{\eta_x\neq\eta^{\prime}_x\}, x\in\partial_{ex}\Lambda.
\end{equation}
Denote by $W_t$ the set of sites explored up to (and including) the integer time $t$, and let $V_t:=\{x\in W_t: S_x=1\}$ be the set of sites with explored value $1$ in $\mathbf{S}$ up to time $t$.
\begin{itemize}
\item Let $W_0:=\partial_{ex}\Lambda$ and $V_0:=\{x\in\partial_{ex}\Lambda: S_x=1\}$.
\item For each $t\geq 0$, reveal the smallest (according to the above ordering) unexplored site $x$ that is adjacent to $V_t$, setting its value in $\mathbf{S}$ via an independent $[0,1]$ uniformly distributed random variable $U_x$:
\begin{align}
&\sigma_x^{\eta}=\begin{cases}
      +1, & U_x\leq P^{\vec{h}}_{\Lambda,\eta,\beta,H}(\sigma_x=+1|\sigma_{W_t}=\sigma^{\eta}_{W_t}), \\
      -1, & \text{otherwise},
   \end{cases}\label{eqreveal1}\\
  & \sigma_x^{\eta^{\prime}}=\begin{cases}
      +1, & U_x\leq P^{\vec{h}}_{\Lambda,\eta^{\prime},\beta,H}(\sigma_x=+1|\sigma_{W_t}=\sigma^{\eta^{\prime}}_{W_t}), \\
      -1, & \text{otherwise},
   \end{cases}\label{eqreveal2}\\
   &S_x=1\{\sigma_x^{\eta}\neq\sigma_x^{\eta^{\prime}}\}.
\end{align}
 Let $W_{t+1}:=W_t\cup\{x\}$, and
\begin{equation}
 V_{t+1}:=\begin{cases}
      V_t\cup\{x\}, & S_x=1,\\
      V_t, & \text{otherwise}.
   \end{cases}
\end{equation}
\item Let $\tau$ be the smallest $t$ at which there is no unexplored site $x$ that is adjacent to $V_t$.
\end{itemize}
Then $V_{\tau}$ is the union of open boundary clusters of $\mathbf{S}$. It is clear that
\begin{equation}
\sigma_x^{\eta}=\sigma_x^{\eta^{\prime}} \text{ for each } x\in (\partial_{ex}V_{\tau})\cap\bar{\Lambda}.
\end{equation}
After exploration time $\tau$, we may reveal the remaining vertices for $\sigma^{\eta}$ and $\sigma^{\eta^{\prime}}$ using the procedure as in \eqref{eqreveal1} and \eqref{eqreveal2}. It is easy to see that $\sigma_x^{\eta}=\sigma_x^{\eta^{\prime}}$ for all those remaining vertices since $(\partial_{ex}V_{\tau})\cap\bar{\Lambda}$ serves as the common boundary condition for $P^{\vec{h}}_{\Lambda,\eta,\beta,H}$ and $P^{\vec{h}}_{\Lambda,\eta^{\prime},\beta,H}$.

Let $sgn(h_x)$ be the sign of $h_x$. Note that if $\vec{H}$ is the bimodal or Gaussian field, $h_x\neq 0$ for all $x\in\Lambda$ almost surely. By worst case analysis, i.e., considering the case where all $2d$ neighbors of $x$ have a sign different from $h_x$, we have
\begin{equation}\label{eqsignS}
P^{\vec{h}}_{\Lambda,\eta,\beta,H}(\sigma^{\eta}_x=sgn(h_x))\geq a(\beta,H,|h_x|):=\frac{e^{-2d\beta+H|h_x|}}{e^{-2d\beta+H|h_x|}+e^{2d\beta-H|h_x|}},
\end{equation}
where the RHS is independent of $\Lambda$ and $\eta$. Therefore $\mathbf{S}$ is stochastically dominated by an inhomogeneous independent site percolation $\mathbf{T}^{\vec{h}}$ (with the same boundary condition as $\mathbf{S}$) with probabilities
\begin{equation}\label{eqdefp}
p_x=p_x(\beta,H,h_x):=1-a^2(\beta,H,|h_x|), x\in\Lambda.
\end{equation}
We emphasize that $p_x$ only depends on $\beta, H$ and $h_x$ and on nothing else. Let $P^{in}_{\Lambda,\vec{h}}$ denote the probability distribution of $\mathbf{T}^{\vec{h}}$. Note that if $\vec{h}\in\mathbb{R}^{\mathbb{Z}^d}$, one can use  \eqref{eqdefp} to define $P^{in}_{\mathbb{Z}^d,\vec{h}}$. We just proved:
\begin{lemma}\label{lemCoupleLF}
Let $\Delta\subseteq \Lambda$ be finite subsets of $\mathbb{Z}^d$. Then
\begin{equation}
\sup_{\eta,\eta^{\prime}\in\{-1,+1\}^{\partial \Lambda}}TV\left(P^{\vec{h}}_{\Lambda,\eta,\beta,H}(\sigma_{\Delta}\in \cdot), P^{\vec{h}}_{\Lambda,\eta^{\prime},\beta,H}(\sigma_{\Delta}\in \cdot)\right)\leq P^{in}_{\Lambda,\vec{h}}(\partial_{ex}\Delta\longleftrightarrow \partial_{ex}\Lambda).
\end{equation}
\end{lemma}

\begin{proof}[Proof of Theorem \ref{thmABB}]
In this case, we have $|h_x|=1$ for each $x\in\mathbb{Z}^d$. So for each $\beta\in[0,\infty)$, by \eqref{eqsignS} and \eqref{eqdefp}, we can choose $H_2(\beta)$ such that for each $H>H_2(\beta)$
\begin{equation}
p_x=1-a^2(\beta,H,1)<p_c^s(d)/2 \text{ for all } x\in\mathbb{Z}^d,
\end{equation}
where  $p_c^s(d)$ is the critical probability for independent Bernoulli site percolation on $\mathbb{Z}^d$.
Then $P^{in}_{\mathbb{Z}^d,\vec{h}}$ is stochastically dominated by an independent Bernoulli site percolation with probability $p_c^s(d)/2$. In particular, there exists $C_4\in(0,\infty)$ such that
\begin{equation}\label{eqindB}
P^{in}_{\mathbb{Z}^d,p}(x\longleftrightarrow y)\leq e^{-C_4|x-y|} \text{ for all } x,y\in\mathbb{Z}^d.
\end{equation}
The theorem now follows from \eqref{eqindB} and Lemma \ref{lemCoupleLF}.
\end{proof}

We next prove Theorem \ref{thmABG}. The proof is more involved than that of Theorem \ref{thmABB}. We assume $H_x\overset{d}{=}N(0,1)$ for each $x\in\mathbb{Z}^d$ in the rest of this subsection. So we can choose $\delta>0$ such that
\begin{equation}\label{eqGaussian}
Prob(|H_x|<\delta)<p_c^s(d)/4 \text{ for each } x\in\mathbb{Z}^d.
\end{equation}
We consider the averaged measure $\bar{P}^{in}_{\mathbb{Z}^d,\vec{H}}$ of the site percolation $\mathbf{T}^{\vec{H}}$ with random field $\vec{H}$ and random probabilities given by
\begin{equation}
p_x(\beta,H,H_x):=1-a^2(\beta,H,|H_x|), x\in\Lambda.
\end{equation}
That is,
\begin{equation}
\bar{P}^{in}_{\mathbb{Z}^d,\vec{H}}(\cdot)=\int_{\mathbb{R}^{\mathbb{Z}^d}}P^{in}_{\mathbb{Z}^d,\vec{h}}(\cdot)Prob(d\vec{h}).
\end{equation}
Then we have
\begin{equation}\label{eqTopen}
\bar{P}^{in}_{\mathbb{Z}^d,\vec{H}}(T^{\vec{H}}(x)=1)=\bar{E}^{in}_{\mathbb{Z}^d,\vec{H}}[p_x(\beta,H,H_x)]\leq Prob(|H_x|<\delta)+1-a^2(\beta,H,\delta).
\end{equation}
This combined with \eqref{eqGaussian} and \eqref{eqsignS} implies that for each $\beta\in[0,\infty)$ there exists $H_3(\beta)$ such that for each $H>H_3(\beta)$
\begin{equation}\label{eqsignSG}
\bar{P}^{in}_{\mathbb{Z}^d,\vec{H}}(T^{\vec{H}}(x)=1)<p_c^s(d)/2 \text{ for all }x\in\mathbb{Z}^d.
\end{equation}
So $\bar{P}^{in}_{\mathbb{Z}^d,\vec{H}}$ is stochastically dominated by an independent Bernoulli site percolation with probability $p_c^s(d)/2$. Thus we proved:
\begin{lemma}\label{lemdecayave}
For each $\beta\in[0,\infty)$ there exists $H_3(\beta)$ such that for each $H>H_3(\beta)$,
\begin{equation}
\bar{P}^{in}_{\mathbb{Z}^d,\vec{H}}(x\longleftrightarrow y)\leq e^{-2C_5|x-y|} \text{ for any } x,y\in\mathbb{Z}^d,
\end{equation}
where $C_5\in(0,\infty)$ depends only on $d$.
\end{lemma}

We next prove that exponential decay is also valid for the quenched measure $P^{in}_{\mathbb{Z}^d,\vec{h}}$.
\begin{lemma}\label{lemdecayque}
For each $\beta\in[0,\infty)$ there exists $H_3(\beta)$ such that for each $H>H_3(\beta)$,
\begin{equation}
P^{in}_{\mathbb{Z}^d,\vec{h}}(x\longleftrightarrow y)\leq C_6(x,\vec{h})e^{-C_5|x-y|} \text{ for any } x,y\in\mathbb{Z}^d\end{equation}
for almost all realizations $\vec{h}\in\{-1,+1\}^{\mathbb{Z}^d}$ of $\vec{H}$, where $C_6(x,\vec{h})\in(0,\infty)$ depends only on $x$ and $\vec{h}$ (and $d$).
\end{lemma}
\begin{proof}
Lemma \ref{lemdecayave} implies that
\begin{equation}
\sum_{y\in\mathbb{Z}^d}e^{C_5|x-y|}\bar{P}^{in}_{\mathbb{Z}^d,\vec{H}}(x\longleftrightarrow y)<\infty.
\end{equation}
By the Fubini-Tonelli theorem,
\begin{equation}
\int_{\mathbb{R}^{\mathbb{Z}^d}}\sum_{y\in\mathbb{Z}^d}e^{C_5|x-y|}P^{in}_{\mathbb{Z}^d,\vec{h}}(x\longleftrightarrow y)dP(\vec{h})<\infty,
\end{equation}
which implies
\begin{equation}
\sum_{y\in\mathbb{Z}^d}e^{C_5|x-y|}P^{in}_{\mathbb{Z}^d,\vec{h}}(x\longleftrightarrow y)<\infty \text{ for almost all }\vec{h}.
\end{equation}
Therefore, there exists $C_6(x,\vec{h})\in(0,\infty)$ such that, for almost all $\vec{h}$,
\begin{equation}
e^{C_5|x-y|}P^{in}_{\mathbb{Z}^d,\vec{h}}(x\longleftrightarrow y)<C_6(x,\vec{h}) \text{ for all } y\in\mathbb{Z}^d,
\end{equation}
which completes the proof.
\end{proof}

\begin{proof}[Proof of Theorem \ref{thmABG}]
This follows immediately from Lemmas \ref{lemCoupleLF} and \ref{lemdecayque}.
\end{proof}

\section{Proof of Corollary \ref{cor2}}\label{secCor}
In this section, we prove Corollary \ref{cor2}.
\begin{proof}[Proof of Corollary \ref{cor2}]
The lower bound follows from the FKG inequality for $P^{\vec{h}}_{\mathbb{Z}^d,\beta,H}$. For the upper bound, we first write
\begin{align}
&\langle \sigma_x\sigma_y\rangle^{\vec{h}}_{\mathbb{Z}^d,\beta,H}-\langle \sigma_x\rangle^{\vec{h}}_{\mathbb{Z}^d,\beta,H}\langle \sigma_y\rangle^{\vec{h}}_{\mathbb{Z}^d,\beta,H}=P^{\vec{h}}_{\mathbb{Z}^d,\beta,H}(\sigma_x=\sigma_y=+1)+P^{\vec{h}}_{\mathbb{Z}^d,\beta,H}(\sigma_x=\sigma_y=-1)\nonumber\\
&\quad-P^{\vec{h}}_{\mathbb{Z}^d,\beta,H}(\sigma_x=+1,\sigma_y=-1)-P^{\vec{h}}_{\mathbb{Z}^d,\beta,H}(\sigma_x=-1,\sigma_y=+1)\nonumber\\
&\quad-[P^{\vec{h}}_{\mathbb{Z}^d,\beta,H}(\sigma_x=+1)-P^{\vec{h}}_{\mathbb{Z}^d,\beta,H}(\sigma_x=-1)][P^{\vec{h}}_{\mathbb{Z}^d,\beta,H}(\sigma_y=+1)-P^{\vec{h}}_{\mathbb{Z}^d,\beta,H}(\sigma_y=-1)].\label{eqtt}
\end{align}
So it suffices to show exponential decay of
\begin{equation}
|P^{\vec{h}}_{\mathbb{Z}^d,\beta,H}(\sigma_x=s_1,\sigma_y=s_2)-P^{\vec{h}}_{\mathbb{Z}^d,\beta,H}(\sigma_x=s_1)P^{\vec{h}}_{\mathbb{Z}^d,\beta,H}(\sigma_x=s_2)|
\end{equation}
for each $s_1, s_2\in\{-1,+1\}$.
Let $L:=|x-y|/2$ and $\Lambda_L(x):=x+\Lambda_L$. Then we have
\begin{align}
&\quad|P^{\vec{h}}_{\mathbb{Z}^d,\beta,H}(\sigma_x=s_1,\sigma_y=s_2)-P^{\vec{h}}_{\mathbb{Z}^d,\beta,H}(\sigma_x=s_1)P^{\vec{h}}_{\mathbb{Z}^d,\beta,H}(\sigma_x=s_2)|\nonumber\\
&\leq |P^{\vec{h}}_{\mathbb{Z}^d,\beta,H}(\sigma_x=s_1|\sigma_y=s_2)-P^{\vec{h}}_{\mathbb{Z}^d,\beta,H}(\sigma_x=s_1)|\nonumber\\
&\leq \sup_{\eta,\eta^{\prime}\in\{-1,+1\}^{\partial_{ex} \Lambda_L(x)}}|P^{\vec{h}}_{\Lambda_L(x),\eta,\beta,H}(\sigma_x=s_1)-P^{\vec{h}}_{\Lambda_L(x),\eta^{\prime},\beta,H}(\sigma_x=s_1)|\label{eqcor2}
\end{align}
by the spatial Markov property of the Ising model. Corollary \ref{cor2} now follows from \eqref{eqtt}, \eqref{eqcor2} and Theorem \ref{thmSBSDB}.
\end{proof}

\section*{Acknowledgements}
The research of JJ was partially supported by STCSM grant 17YF1413300 and that of CMN by US-NSF grant DMS-1507019. The authors thank Dan Stein and Janek Wehr for useful discussions. They also thank the Institute of Applied Mathematics of the Chinese Academy of Sciences, where some of the work reported here was done.

\end{document}